\documentclass[a4paper,12pt,intlimits,oneside]{amsart}
\usepackage{amsmath}
\usepackage{amsthm}
\usepackage{latexsym}
\usepackage{amssymb}
\usepackage{xcolor}
\usepackage{graphicx}
\usepackage{mathrsfs}
\numberwithin{figure}{section}
\def\R{{\mathbb R}}
\def\C{{\mathbb C}}

\def\T{{\mathbb T}}

\def\la{\langle}
\def\ra{\rangle}
\def\s{\vskip 0.25cm\noindent}

\def\build#1_#2^#3{\mathrel{
\mathop{\kern 0pt#1}\limits_{#2}^{#3}}}
\def\td_#1,#2{\mathrel{\mathop{\build\longrightarrow_{#1\rightarrow #2}^{}}}}
\DeclareFontFamily{U}{MnSymbolC}{}
\DeclareSymbolFont{MnSyC}{U}{MnSymbolC}{m}{n}
\DeclareFontShape{U}{MnSymbolC}{m}{n}{
    <-6>  MnSymbolC5
   <6-7>  MnSymbolC6
   <7-8>  MnSymbolC7
   <8-9>  MnSymbolC8
   <9-10> MnSymbolC9
  <10-12> MnSymbolC10
  <12->   MnSymbolC12}{}
\DeclareMathSymbol{\intprod}{\mathbin}{MnSyC}{'270}
\newtheorem{theorem}{Theorem}
\newtheorem{corollary}{Corollary}

\newtheorem{lemma}{Lemma}

\begin{document}
\title[Lax pair for sBO]{The Lax pair structure\\ for the spin Benjamin--Ono equation}
\author[P. G\'erard]{Patrick G\'erard}
\address{Laboratoire de Math\'ematiques d'Orsay,  CNRS, Universit\'e Paris--Saclay, 91405 Orsay, France} \email{{\tt patrick.gerard@universite-paris-saclay.fr}}

\begin{abstract}
We prove that the recently introduced spin Benjamin--Ono equation admits a Lax pair, and we deduce a family of conservation laws which allow to prove global wellposedness in all Sobolev spaces $H^k$ for every integer $k\geq 2$. We also infer an additional family of matrix valued conservation laws, of which the previous family are just the traces.
\end{abstract}


\date{February 16,  2022}

\thanks {The author is grateful to Edwin Langmann for drawing his attention to equation \eqref{sBO} and for stimulating discussions.}

\maketitle

\section{Introduction}
In a recent paper \cite{BLL}, Berntson, Langmann and Lenells have introduced the following spin generalization of the Benjamin--Ono equation on the line $\R $ or on the torus $\T $,
$$\partial_tU+\{U,\partial_x U\} +H\partial_x^2U -i[U,H\partial_xU]=0\ ,\ x\in X,$$
where $X$ denotes $\R $ or $\T$, the unknown $U$ is valued into  $d\times d$ matrices, and $H$ denotes the scalar Hilbert transform on $X$ ; in fact, the authors chose the  normalization  $H=i\, {\rm sign}(D)$, so that $H\partial_x =-|D|$, where $|D|$ denotes the Fourier multiplier associated to the symbol $|k|$. Notice that, in front of the commutator term in the right hand side, we take a different sign from the one used in \cite{BLL}. However it is easy to pass to the other sign by applying the complex conjugation. Consequently, the above equation reads
\begin{equation}\label{sBO}
\partial_tU=\partial_x (|D|U-U^2)-i[U,|D|U]\ .
\end{equation}
 The purpose of this note is to prove that equation \eqref{sBO} enjoys a Lax pair structure, and to infer first consequences on the corresponding dynamics. 
\section{The Lax pair structure}
Let us first introduce some more notation. Given operators $A,B$, we denote
$$\{ A,B\} :=AB+BA\ ,\ |A,B]:=AB-BA\ $$
and $A^*$ denote the adjoint of $A$.
 We consider the Hilbert space $\mathscr H:=L^2_+(X,\C ^{d\times d})$ made of $L^2$ functions on $X$  with Fourier transforms supported in nonnegative modes, and valued into $d\times d$ matrices, endowed with the inner product
 $\la A\vert B\ra =\int_X {\rm tr}(AB^*)\, dx\ .$
  We denote by $\Pi_{\geq 0}$ the orthogonal projector from $L^2(X,\C^{d\times d})$ onto $\mathscr H$.
\s
According to the study of the integrability of the scalar Benjamin--Ono equation \cite{GK}, given $U\in L^2(X,\C^{d\times d})$ valued into $\C^{d\times d}$, we define on $\mathscr H$ the unbounded operator
$$L_U:=D-T_U\ ,\ D:=\frac 1i \partial_x\ ,$$
where ${\rm dom}(L_U):=\{ F\in \mathscr H: DF\in \mathscr H\}$, and $T_U$ is the Toeplitz operator of symbol $U$ defined by  $T_U(F):=\Pi_{\geq 0}(UF)$ . 
\s
It is easy to check that $L_U$ is selfadjoint if $U$ is valued in Hermitian matrices. However we do not need the latter property for establishing the Lax pair structure. If $U$ is smooth enough (say belonging to the Sobolev space $H^2$), we define the following bounded operator,
$$B_U:=i(T_{|D|U}-T_U^2)\ ,$$
which is antiselfadjoint if $U$ is valued in Hermitian matrices. Our main result is the following.
\begin{theorem}\label{laxpair} 
Let  $I$ be a time interval, and $U$ be a continuous function on $I$ valued  into  $H^2(X,\C^{d\times d})$ such that $\partial_tU$ is continuous valued into  $L^2(X,\C^{d\times d})$. Then $U$ is a solution of \eqref{sBO} on $I$ if and only if
$$\partial_tL_U=[B_U,L_U]\ .$$
\end{theorem}
\begin{proof}
Obviously, $\partial_tL_U=-T_{\partial_tU}$. Since $T_G=0$ implies classically $G=0$, the claim is equivalent to the identity
$$-T_{\partial_x (|D|U-U^2)-i[U,|D|U]}=[B_U,L_U]\ .$$
We have 
\begin{eqnarray*}
-T_{\partial_x (|D|U-U^2)-i[U,|D|U]}
&=&[iT_{|D|U},D]+T_{U\partial_xU+\partial_xU\, U}+iT_{[U,|D|U]}\\
&=&[B_U,D]+T_{U\partial_xU+\partial_xU\,U}-T_UT_{\partial_xU}-T_{\partial_xU}T_U+iT_{[U,|D|U]}\\
&=&[B_U,L_U]+T_{\{ U,\partial_xU\}} -\{ T_U,T_{\partial_xU}\}+iT_{[U,|D|U]}-i[T_U,T_{|D|U}] \end{eqnarray*}
So we have to check that
\begin{equation}\label{critical}
T_{\{ U,\partial_xU\}} -\{ T_U,T_{\partial_xU}\}+iT_{[U,|D|U]}-i[T_U,T_{|D|U}] =0\ .
\end{equation}
We need the following lemma, where we denote $\Pi_{<0}:=Id -\Pi_{\geq 0}\ .$ 
\begin{lemma}\label{Tcom}
Let $A, B\in L^\infty(X,\C^{d\times d})$. Then, for every $F\in \mathscr H$,
$$(T_{AB}-T_AT_B)F=\Pi_{\geq 0}(\Pi_{\geq 0}(A)\, \Pi_{<0}(\Pi_{<0}(B)F))\ .$$
\end{lemma}
Let us prove Lemma \ref{Tcom}. Write
$$T_{AB}F=\Pi_{\geq 0}(ABF)=\Pi_{\geq 0}(A\Pi_{\geq 0}(BF))+\Pi_{\geq 0}(A\Pi_{<0}(BF))=T_AT_BF+\Pi_{\geq 0}(A\Pi_{<0}(BF)), $$
so that, observing that the ranges of $\Pi _{\geq 0}$ and of $\Pi_{<0}$ are stable through the multiplication,
$$(T_{AB}-T_AT_B)F=\Pi_{\geq 0}(A\Pi_{<0}(BF))=\Pi_{\geq 0}(\Pi_{\geq 0}(A)\Pi_{<0}(\Pi_{< 0}(B)F))\ .$$
This completes the proof of Lemma \ref{Tcom}. 
\s
Let us apply Lemma \ref{Tcom} to $A=U$, $B=|D|U$. We get
\begin{eqnarray*}i(T_{U|D|U}-T_UT_{|D|U})F&=&\Pi_{\geq 0}(\Pi_{\geq 0}(U)\, \Pi_{<0}(\Pi_{<0}(i|D|U)F))\\
&=&-\Pi_{\geq 0}(\Pi_{\geq 0}(U)\, \Pi_{<0}(\Pi_{<0}(\partial_xU)F))\ ,
\end{eqnarray*}
and similarly
\begin{eqnarray*}i(T_{|D|U\,  U}-T_{|D|U}T_{U})F&=&\Pi_{\geq 0}(\Pi_{\geq 0}(i|D|U)\, \Pi_{<0}(\Pi_{<0}(U)F))\\
&=&\Pi_{\geq 0}(\Pi_{\geq 0}(\partial_xU)\, \Pi_{<0}(\Pi_{<0}(U)F))\ ,
\end{eqnarray*}
so that
\begin{eqnarray*}
(iT_{[U,|D|U]}-i[T_U,T_{|D|U}])F&=&-\Pi_{\geq 0}(\Pi_{\geq 0}(U)\, \Pi_{<0}(\Pi_{<0}(\partial_xU)F))\\
&&-\Pi_{\geq 0}(\Pi_{\geq 0}(\partial_xU)\, \Pi_{<0}(\Pi_{<0}(U)F))\\
&=&-T_{\{ U,\partial_xU\}} (F)+\{ T_U,  T_{\partial_xU}\} (F)\ ,
\end{eqnarray*}
using again Lemma \ref{Tcom}. Hence we have proved identity \eqref{critical}.
\end{proof}
\section{Conservation laws and global wellposedness}
The following is an application of Theorem \ref{laxpair}.
\begin{corollary}\label{gwp}
Assume $U_0$ belongs to the Sobolev space $H^2(X,\C^{d\times d})$, and is valued into Hermitian matrices. Then equation \eqref{sBO} has a unique solution $U$, depending continuously of $t\in \R$, valued into  Hermitian matrices of the Sobolev space $H^2(X)$,  and such that $U(0)=U_0$. Furthermore,  the following quantities are conservation laws,
$$\mathscr E_k(U)=\la L_U^k(\Pi_{\geq 0} U)\vert \Pi_{\geq 0}U\ra \ ,\ k=0,1,2\dots $$
In particular, the norm of $U(t)$ in the Sobolev space $H^2(X)$ is uniformly bounded for $t\in \R$.
\end{corollary}
\begin{proof}  
The local wellposedness in the Sobolev space $H^2$ follows from an easy adaptation of  Kato's iterative scheme --- see e.g. Kato \cite{K} for hyperbolic systems. Global wellposedness will follow if we show that conservation laws control the $H^2$ norm.
Set $U_+:=\Pi_{\geq 0}U\ ,\ U_-:=\Pi_{<0}U\ .$ Applying $\Pi_{\geq 0}$ to both sides of \eqref{sBO}, we get
$$\partial_tU_+=-i\partial_x^2U_+-2T_U\partial_xU_+-2T_{\partial_xU_-}U_+=iL_U^2(U_+)+B_U(U_+)\ .$$
Therefore, from Theorem \ref{laxpair}, 
\begin{eqnarray*}
\frac{d}{dt}\la L_u^k (U_+)\vert U_+\ra &=&\la [B_U,L_U^k]U_+\vert U_+\ra +\la L_U^k(iL_U^2(U_+)+B_U(U_+))\vert U_+\ra +\\&& +\la L_U^k(U_+)\vert iL_U^2(U_+)+B_U(U_+)\ra \\
&=& 0\ ,
\end{eqnarray*}
since $B_U$ and $iL_U^2$ are antiselfajoint. \\
Now observe that $\mathscr E_0(U)=\| U_+\|_{L^2}^2$. Since $U$ is Hermitian, we have
$$U=  \begin{cases} U_++ U_+^* &{\rm if}\  X=\R \  ,\\
U_++U_+^* -\la U_+\ra &{\rm if}\  X= \T \ , 
\end{cases} $$
where $\la F\ra $ denotes the mean value of a function $F$ on $\T $. We infer that $\mathscr E_0(U)$ controls the $L^2$ norm of $U$.   Let us come to $\mathscr E_1(U)$. In view of the Gagliardo--Nirenberg inequality,
\begin{eqnarray*}
\mathscr E_1(U)&=&\la DU_+\vert U_+\ra -\la T_U(U_+)\vert U_+\ra \geq \la DU_+\vert U_+\ra -O(\| U_+\| _{L^3}^3)\\
&\geq & \la DU_+\vert U_+\ra- O(\la DU_+\vert U_+\ra ^{1/2}\| U_+\|_{L^2}^{2})-O(\| U_+\|_{L^2}^3)\ .
\end{eqnarray*}
Consequently, $\mathscr E_0(U)$ and $\mathscr E_1(U)$ control $\| U_+\|_{L^2}^2+\la DU_+\vert U_+\ra $, which is the square of the $H^{1/2} $ norm of $U_+$, since $U_+$ only has nonnegative Fourier modes. Therefore the $H^{1/2}$ norm of $U$ is controlled by $\mathscr E_0(U)$ and $\mathscr E_1(U)$.\\ Since 
$\mathscr E_2(U)$ is the square of $L^2$ norm of $L_U(U_+)$, and since the $L^2$ norm of $T_U(U_+)$ is controlled by the $H^{1/2}$ norm of $U$ by the Sobolev estimate, we infer that $\mathscr E_0(U)$, $\mathscr E_1(U)$ and $\mathscr E_2(U)$ control the $L^2$ norms of $U$ and of $\partial_xU$, namely the Sobolev $H^1$ norm of $U$.
\\ Finally, $\mathscr E_4(U)$ is the square if the $L^2$ norm of $L_U^2(U_+)$. Since $L_U(U_+)$ is already controlled in $L^2$ and $U$ is controlled in $L^\infty $ by the Sobolev inclusion $H^1\subset L^\infty $, we infer that the $H^1$ norm of $L_U(U_+)$ is controlled. But $H^1$ is an algebra, so the $H^1$ norm of $T_U(U_+)$ is also controlled. Finally, we infer that $\{ \mathscr E_n(U), n\leq 4\}$  control the  $H^1$ norms of $U_+$ and $\partial_xU_+$, namely the $H^2$ norm of $U_+$, and finally of $U$.
\end{proof}
{\bf Remarks.}
\begin{enumerate}
\item If the initial datum $U$ belongs to the Sobolev space $H^k$ for an integer $k>2$, a similar argument shows that the $H^k$ norm of $U$ is controlled by the collection  $\{ \mathscr E_n (U), 0\leq n\leq 2k\}$.
\item In \cite{BLL}, the evolution of multi--solitons for \eqref{sBO} is derived through a pole ansatz, and the question of keeping the poles  away from the real line --- or from the unit circle in the case $X=\T $--- is left open. Since Corollary  \ref{gwp} implies that the $L^\infty $ norm of the solution stays bounded as $t$ varies, this implies a positive answer to this question, as far as the poles do not collide. In fact, we strongly suspect that such a collision does not  affect the structure of the pole ansatz, because it is likely that multisolitons have a characterization in terms of the spectrum of $L_U$, as it has in the scalar case \cite{GK}.
\end{enumerate}
Let us say a few more about the conservation laws. The conservation laws $\mathscr E_k$ can be explicitly computed in terms of $U$. For simplicity, we focus on $\mathscr E_0$ and $\mathscr E_1$. In the case $X=\R $, we have exactly 
$$\mathscr E_0(U)=\frac 12 \int _\R {\rm tr}(U^2)\, dx\ ,$$
and
\begin{eqnarray*}
 \mathscr E_1(U)&=&\la DU_+\vert U_+\ra -\la T_U(U_+)\vert U_+\ra \\
 &=&\int_\R {\rm tr} \left ( \frac 12 U|D|U -\frac 13 U^3\right )\, dx \ ,
 \end{eqnarray*}
 so we recover  the Hamiltonian function derived in \cite{BLL}.\\
  In the case $X=\T $, the above formulae must be slightly modified due the zero Fourier mode. This leads us to a {\bf bigger set of conservation laws}. Indeed, every constant matrix $V\in \C^{d\times d}$ is a special element of $\mathscr H$, and we observe that $B_U(V)=-iL_U^2(V)\ .$
  Arguing exactly as in the proof of Corollary \ref{gwp}, we infer that, for every integer $\ell \geq 1$, for every pair of constant matrices $V,W$, the quantity 
  $\la L_U^\ell (V)\vert W\ra $ is a conservation law. Since $V,W$ are arbitrary, this means that, if ${\bf 1}$ denotes the identity matrix,  all the matrix--valued functionals
  $$\mathscr M_{\ell-2}(U):=\int_{\T} L_U^{\ell} ({\bf 1})\, dx$$
  for $\ell \geq 1$, are conservation laws. If the measure of $\T $ is normalised to $1$, we have for instance
    \begin{eqnarray*}
    \mathscr M_{-1}(U)&=&-\la U_+\ra =-\la U\ra \ ,\\ 
    \mathscr M_0(U)&=&\frac 12\la U^2-iUHU\ra +\frac 12\la U\ra ^2\ .
   \end{eqnarray*}  
  Then one can check that
  \begin{eqnarray*}
  \mathscr E_0(U)&=&   \frac 12{\rm tr}(\la U^2\ra )+\frac 12 {\rm tr}(\la U\ra ^2)\ ,\\
 \mathscr E_1(U)&=& {\rm tr }\left \la \frac 12 U|D|U -\frac 13 U^3\right \ra -\frac 53 {\rm tr} [\la U \ra ^3] -{\rm tr}[\mathscr M_0(U) \la U\ra ]\ .
  \end{eqnarray*} 
 Observe again that the first term in the right hand side of the expression of $\mathscr E_1(U)$ is the opposite of the Hamiltonian function in \cite{BLL}.\\
In the case $X=\R$, all the matrix valued expressions $\mathscr M_k(U)$ make sense if $k\geq 0$ and are again conservation laws. For instance,
 $$\mathscr M_0(U)=\frac 12\int_\R( U^2-iUHU)\, dx\ .$$ 
 Finally, notice that in both cases $X=\T$ and $X=\T$, we have $$\mathscr E_k(U)={\rm tr}\mathscr M_{k}(U)$$ for every $k\geq 0$.


\end{document}